\documentclass[12pt, reqno]{amsart}

\newcommand\version{March 18, 2020}

\usepackage[parfill]{parskip}    
\usepackage{graphicx}
\usepackage{amsmath,amsfonts,amsthm,amssymb,amsxtra}
\usepackage{bbm} 
\usepackage{mathabx}
\usepackage{commath}
\usepackage{epstopdf}

\usepackage[T1]{fontenc}
\usepackage[utf8]{inputenc}

\usepackage[ngerman, german, english]{babel}
\usepackage{bbm}

\newcommand{\1}{\mathbbm{1}}
\newcommand{\R}{\mathbb{R}} 
\newcommand{\N}{\mathbb{N}} 
\newcommand*\diff{\mathop{}\!\mathrm{d}} 
\newcommand{\Sph}{{{\mathbb S}^n}}
\newcommand{\Ste}{{\mathcal S}}
\newcommand{\Slxi}{{\Sigma_{\lambda, \xi_0}}}
\newcommand{\lx}{{\lambda, x_0}}
\newcommand{\lxi}{{\lambda, \xi_0}}

\DeclareMathOperator{\supp}{supp}

\newtheorem{theorem}{Theorem}
\newtheorem{proposition}[theorem]{Proposition}

\newtheorem{lemma}[theorem]{Lemma}

\newtheorem*{definition*}{Definition}

\newtheorem{remark}[theorem]{Remark}

\begin{document}

\title[Classification of solutions --- \version]{Classification of solutions of an  equation\\ related to a conformal log Sobolev inequality}

\author{Rupert L. Frank}

\address[Rupert L. Frank]{Mathematisches Institut, Ludwig-Maximilians-Universit\"at M\"unchen, Theresienstr. 39, 80333 M\"unchen, Germany, and Mathematics 253-37, Caltech, Pasa\-de\-na, CA 91125, USA}

\email{r.frank@lmu.de, rlfrank@caltech.edu}

\author{Tobias K\"onig}

\address[Tobias K\"onig]{Mathematisches Institut, Ludwig-Maximilians-Universit\"at M\"unchen, Theresienstr. 39, 80333 M\"unchen, Germany}

\email{tkoenig@math.lmu.de}

\author{Hanli Tang}

\address[Hanli Tang]{School of Mathematical Sciences\\Beijing Normal University\\Beijing, 100875, China}

\email{hltang@bnu.edu.cn}

\thanks{\copyright\, 2020 by the authors. This paper may be reproduced, in its entirety, for non-commercial purposes.\\
	The first author is grateful to M.~Zhu for a correspondence in May 2012 on the topic of this paper.\\
	Partial support through US National Science Foundation grant DMS-1363432 (R.L.F.), Studienstiftung des deutschen Volkes (T.K.)  and National Natural Science Foundation of China (Grant No.11701032) (H.L.T.) is acknowledged.}

\begin{abstract}
	We classify all finite energy solutions of an equation which arises as the Euler--Lagrange equation of a conformally invariant logarithmic Sobolev inequality on the sphere due to Beckner. Our proof uses an extension of the method of moving spheres from $\R^n$ to $\Sph$ and a classification result of Li and Zhu. Along the way we prove a small volume maximum principle and a strong maximum principle for the underlying operator which is closely related to the logarithmic Laplacian.
\end{abstract}

\maketitle

\section{Introduction}

\subsection{Main result}

The motivation of this paper is Beckner's logarithmic Sobolev inequality on $\Sph$ with sharp constant \cite{Be1992,Be1997}. It states that
\begin{equation}
\label{beckner log sob}
\iint_{\Sph \times \Sph} \frac{|v(\omega) - v(\eta)|^2}{|\omega - \eta|^n} \, \diff\omega \diff\eta  \geq C_n \int_{\Sph} |v(\omega)|^2 \ln \frac{|v(\omega)|^2 |\Sph|}{\|v\|_2^2} \diff\omega
\end{equation}
with
\begin{equation}
\label{eq:becknerconst}
C_n = \frac{4}{n} \frac{\pi^{n/2}}{\Gamma(n/2)} \,.
\end{equation}
Here and in the following, $\diff \omega$ denotes the surface measure induced by the embedding of $\Sph$ in $\R^{n+1}$, i.e., $\int_\Sph \diff \omega = |\Sph| = 2 \pi^\frac{n+1}{2} / \Gamma (\frac{n+1}{2} )$.

Note that, by Jensen's inequality and convexity of $x\mapsto x\ln x$, the right side of \eqref{beckner log sob} is nonnegative and vanishes if and only if $|v|$ is constant. Inequality \eqref{beckner log sob} is a limiting form of the Sobolev inequalities and, in the spirit of these inequalities, it states that functions with some regularity (quantified by the finiteness of the left side) have some improved integrabiliy properties (quantified by the finiteness of the right side). Beckner used inequality \eqref{beckner log sob} to prove an optimal hypercontractivity bound for the Poisson semigroup on the sphere. A remarkable feature of inequality \eqref{beckner log sob} is its conformal invariance, which we will discuss below in detail.

In \cite{Be1997} Beckner showed that equality holds in \eqref{beckner log sob} if and only if
\begin{eqnarray}
\label{eq:beckneropt}
v(\omega) = c \left( \frac{\sqrt{1-|\zeta|^2}}{1-\zeta\cdot\omega} \right)^{n/2}
\end{eqnarray}
for some $\zeta\in\R^{n+1}$ with $|\zeta|<1$ and some $c\in\R$.

Our goal in this paper is to classify all nonnegative solutions $u$ of the equation
\begin{equation}
\label{equation 1}
P.V. \int_{\Sph} \frac{u(\omega) - u(\eta)}{|\omega - \eta|^n} \diff \eta = C_n u(\omega) \ln u(\omega)
\qquad\text{in}\ \Sph.
\end{equation}
This equation arises, after a suitable normalization, as the Euler--Lagrange equation of the optimization problem corresponding to \eqref{beckner log sob}.

Because of the principal value in \eqref{equation 1} we interpret this  equation in the weak sense. The maximal class of functions for which \eqref{beckner log sob} holds is
$$
\mathcal D := \left\{ v \in L^2(\Sph) :\ \iint_{\Sph \times \Sph} \frac{|v(\omega) - v(\eta)|^2}{|\omega - \eta|^n}\, \diff\omega\diff\eta <\infty \right\}.
$$

We say that a nonnegative function $u\in\mathcal D$ on $\Sph$ is a \emph{weak solution} of \eqref{equation 1} if
$$
\frac{1}{2}\iint_{\Sph\times\Sph}\frac{(\varphi(\omega)-\varphi(\eta))\,(u(\omega)-u(\eta))}{|\omega-\eta|^n}\,\diff\omega\diff\eta = C_n \int_\Sph \varphi(\omega)\, u(\omega) \ln u(\omega)  \diff \omega
$$
for every $\varphi \in \mathcal D$.

Clearly, the constant function $u\equiv 1$ is a weak solution of \eqref{equation 1}. Because of the conformal invariance, which equation \eqref{equation 1} inherits from inequality \eqref{beckner log sob}, see Lemma \ref{lemma conf inv}, all elements in the orbit of the constant function $u\equiv 1$ under the conformal group are also weak solutions. One can show that these are precisely the functions of the form \eqref{eq:beckneropt} with $c=1$. Our main result is that these are \emph{all} the finite energy solutions of \eqref{equation 1}.

\begin{theorem}
	\label{theorem classification}
	Let $0 \nequiv u\in \mathcal D$ be a nonnegative weak solution of equation \eqref{equation 1}. Then
	\[ u (\omega) = \left( \frac{\sqrt{1- |\zeta|^2}}{1 - \zeta \cdot \omega} \right)^{n/2}  \]
	for some $\zeta \in \R^{n+1}$ with $|\zeta| < 1$.
\end{theorem}

As we will explain in the next subsection, this result and its proof are in the spirit of similar classification results for conformally invariant equations. Groundbreaking results in the local case were obtained by Gidas, Ni and Nirenberg \cite{GiNiNi} and Caffarelli, Gidas and Spruck \cite{CaGiSp}.
In the nonlocal case similar results were first obtained by Chen, Li and Ou \cite{ChLiOu} and Li \cite{Li04} and we refer to these works for further references.

We follow a general strategy that was pioneered by Li and Zhu \cite{LiZh}; see also \cite{LiZha}. The basic observation there is that a symmetry result together with the conformal invariance of the equation forces the solutions to be of the claimed form. More precisely, the proof proceeds in two steps. In a first step one uses the method of moving planes or its variant, the method of moving spheres, in order to show symmetry of positive solutions. The symmetry in question respects the conformal invariance of the equation. The second step employs a powerful lemma by Li and Zhu \cite{LiZh} which classifies the sufficently regular functions which have the conformal symmetry property established in the first step.

The adaptation of these methods to the present setting, however, encounters several difficulties. One of these comes from the fact that functions in $\mathcal D$ have only a very limited regularity. In fact, the left side of \eqref{beckner log sob} is comparable to
$$
\left(v,\left(\ln(-\Delta_{\Sph}+1)\right) v\right)
$$
with the inner product in $L^2(\Sph)$, see \eqref{eq:energyharmonics} below. Thus, the linear operator on the left side of \eqref{equation 1} is reminiscent of the logarithmic Laplacian, studied recently in \cite{ChWe} on a domain in Euclidean space; see also \cite{DLNoPo} for a related functional. The work \cite{ChWe} contains some regularity results, but we have not been able to use these to deduce that solutions $u$ of \eqref{equation 1} are continuous. Therefore, we need to perform the method of moving spheres in the energy space. While this can be carried out in an elegant and concise way in the case of $(-\Delta)^{\pm s}$ \cite{ChLiOu}, the proof of the corresponding small volume maximum principle in our setting is rather involved and constitutes one of the main achievements in this paper; see Section \ref{section max principles}. The missing regularity also prevents us from directly applying the classification lemma by Li and Zhu \cite{LiZh}. Instead, we use its extension in \cite{FrLi} to measures; see Section \ref{section proof main}.

We believe that the techniques that we develop in this paper can be useful in similar problems and that they illustrate, in particular, how to prove classification theorems in problems with conformal invariance without first establishing regularity results.


\subsection{Background}

In order to put this problem into context, let us recall Lieb's sharp form \cite{Li} of the Hardy--Littlewood--Sobolev inequality, which states that, if $0<\lambda<n$, then for any $f\in L^{2n/(2n-\lambda)}(\R^n)$,
\begin{equation}\label{eq:hls}
\iint_{\R^n\times\R^n} \frac{f(x)\,f(y)}{|x-y|^\lambda}\diff x\diff y \leq \mathcal C_{\lambda,n} \left( \int_{\R^n} |f|^{2n/(2n-\lambda)}\diff x \right)^{(2n-\lambda)/n}
\end{equation}
with
$$
\mathcal C_{\lambda,n} = \pi^{\lambda/2} \frac{\Gamma(\frac{n-\lambda}2)}{\Gamma(n+\frac \lambda 2)} \left( \frac{\Gamma(n)}{\Gamma(\frac n2)} \right)^{1-\lambda/n} \,.
$$
Moreover, equality in \eqref{eq:hls} holds if and only if
\begin{eqnarray}
\label{eq:hlsopt}
f(x) = c \left( \frac {2b}{b^2 + |x-a|^2}\right)^{(2n-\lambda)/2}
\end{eqnarray}
for some $a\in\R^n$, $b>0$ and $c\in\R$. The Euler--Lagrange equation of the optimization problem related to \eqref{eq:hls} reads, in a suitable normalization,
\begin{equation}\label{eq:hlseq}
\int_{\R^n} \frac{f(y)}{|x-y|^\lambda}\diff y = |f(x)|^{-2(n-\lambda)/(2n-\lambda)} f(x)
\qquad\text{in}\ \R^n \,.
\end{equation}
Lieb posed the classification of positive solutions of \eqref{eq:hlseq} as an open problem, which was finally solved by Chen, Li and Ou \cite{ChLiOu} and Li \cite{Li04}. They showed that the only positive solutions in $L^{2n/(2n-\lambda)}_{\rm loc}(\R^n)$ of \eqref{eq:hlseq} are given by \eqref{eq:hlsopt} with $a\in\R^n$, $b>0$ and with a constant $c$ depending only on $\lambda$ and $n$.

Writing $|x-y|^{-\lambda}$ in \eqref{eq:hls} as a constant times $\int_{\R^n} |x-z|^{-(n+\lambda)/2} |z-y|^{-(n+\lambda)/2}\diff z$ and recognizing $|\cdot|^{-(n+\lambda)/2}$ as a constant times the Green's function of $(-\Delta)^{(n-\lambda)/4}$, we see by duality, putting $\lambda = n - 2s$, that \eqref{eq:hls} is equivalent to the sharp Sobolev inequality, namely, if $0 < s < n/2$, then for all $u\in \dot H^s(\R^n)$,
\begin{equation}
\label{eq:sob}
\|(-\Delta)^{s/2} u \|_2^2 \geq \mathcal S_{s,n} \left( \int_{\R^n} |u|^{2n/(n-2s)}\diff x \right)^{(n-2s)/n}
\end{equation}
with
$$
\mathcal S_{s,n} = (4\pi)^s \ \frac{\Gamma(\frac{n+2s}{2})}{\Gamma(\frac{n-2s}{2})} \left( \frac{\Gamma(\frac n2)}{\Gamma(n)} \right)^{2s/n}
= \frac{\Gamma(\frac{n+2s}{2})}{\Gamma(\frac{n-2s}{2})} \ |\Sph|^{2s/n} \,.
$$
Moreover, equality holds if and only if
$$
u(x) = c \left( \frac {2b}{b^2 + |x-a|^2}\right)^{(n-2s)/2} \,.
$$
By integrating the Euler--Lagrange equation corresponding to \eqref{eq:sob} against $|x-y|^{-(n-2s)}$, we obtain \eqref{eq:hlseq} with $f$ replaced by a multiple of $|u|^{4s/(n-2s)}u$. This leads to a classification of all positive solutions in $\dot H^s(\R^n)$ of the corresponding Euler--Lagrange equation \cite{ChLiOu}.

A crucial step in Lieb's proof of the sharp inequality \eqref{eq:hls} and the classification of its optimizers was the observation that it is equivalent to the following sharp inequality on $\Sph$,
\begin{equation}\label{eq:hlssphere}
\iint_{\Sph\times\Sph} \frac{g(\omega)\,g(\eta)}{|\omega-\eta|^\lambda}\diff \omega\diff \eta \leq \mathcal C_{\lambda,n} \left( \int_{\Sph} |g|^{2n/(2n-\lambda)}\diff \omega \right)^{(2n-\lambda)/n}.
\end{equation}
In fact, each side of \eqref{eq:hlssphere} equals the corresponding side in \eqref{eq:hls} if
$$
f(x) = \left( \frac{2}{1+|x|^2} \right)^{(2n-\lambda)/2} g(\mathcal S(x)) \,,
$$
where $\mathcal S:\R^n\to\Sph$ is the inverse stereographic projection; see \eqref{eq:stereo} below. This transformation yields also a characterization of optimizers and of positive solutions of the Euler--Lagrange equation corresponding to \eqref{eq:hlssphere}. The functions $f$ in \eqref{eq:hlsopt} become
\begin{equation}
\label{eq:optsphere}
g(x) = c \left( \frac{\sqrt{1-|\zeta|^2}}{1-\zeta\cdot\omega} \right)^{(2n-\lambda)/2}
\end{equation}
with $\zeta\in\R^{n+1}$ such that $|\zeta|<1$. More explicitly, there is a bijection between such $\zeta$ and parameters $a\in\R^n$, $b>0$ in \eqref{eq:hlsopt} given by $\zeta = (2\eta - b^2(1+\eta_{n+1})e_{n+1})/(2+b^2(1+\eta_{n+1}))$ with $\eta = \mathcal S(a)$.

Beckner \cite[Eq.~(19)]{Be1993} observed that, in the same sense as \eqref{eq:sob} is the dual of \eqref{eq:hls}, the dual of \eqref{eq:hlssphere} is
\begin{eqnarray}
\label{eq:sobsphere}
\left\| A_{2s}^{1/2} v \right\|_2^2 \geq \mathcal S_{s,n} \|v\|_q^2
\end{eqnarray}
with
\begin{align}
\label{eq:opas}
A_{2s} = \frac{\Gamma(B+\tfrac12 + s)}{\Gamma(B+\tfrac12 - s)}
\qquad\text{and}\qquad
B = \sqrt{-\Delta_{\Sph} + \tfrac{(n-1)^2}{4}} \,.
\end{align}
The operators $A_{2s}$ are special cases of the GJMS operators in conformal geometry \cite{GrJeMaSp}. The duality between \eqref{eq:hlssphere} and \eqref{eq:sobsphere} and the known results about the former yield a characterization of optimizers and of positive solutions of the Euler--Lagrange equation corresponding to \eqref{eq:sobsphere}.

The relation between these inequalities and classification results and the problem studied in this paper is as follows. Inequality \eqref{eq:sobsphere} becomes an equality as $s\to 0$. Differentiating at $s=0$, Beckner \cite{Be1992} obtained the inequality
$$
\left( v,\left(\psi(B+\tfrac12)-\psi(\tfrac n2) \right) v \right) \geq \frac1n \int_{\Sph} |v(\omega)|^2 \ln \frac{|v(\omega)|^2 |\Sph|}{\|v\|_2^2} \diff\omega \,,
$$
where $\psi=\Gamma'/\Gamma$ is the digamma function. Using the Funk--Hecke formula one can show that
\begin{equation}
\label{eq:energyharmonics}
\left( v,\left(\psi(B+\tfrac12)-\psi(\tfrac n2) \right) v \right) = \frac{1}{n\,C_n} \iint_{\Sph \times \Sph} \frac{|v(\omega) - v(\eta)|^2}{|\omega - \eta|^n} \, \diff\omega \diff\eta \,,
\end{equation}
which yields \eqref{beckner log sob}. Alternatively, one can subtract
$$
\int_\Sph \frac{d\omega}{|\omega-e|^\lambda} \|g\|_2^2
$$
(with $e\in\Sph$ arbitrary) from the left side of \eqref{eq:hlssphere} and pass to the limit $\lambda\to n$. From the characterization of optimizers in \eqref{eq:hlssphere} or \eqref{eq:sobsphere} (or by a simple computation), one finds that the functions in \eqref{eq:beckneropt} are optimizers in \eqref{beckner log sob}. Because of the limiting argument, however, uniqueness of these optimizers requires a separate argument \cite{Be1997}.

Similarly, characterization of the solutions of the Euler--Lagrange equations corresponding to \eqref{eq:hlssphere} or \eqref{eq:sobsphere} does not yield the characterization of solutions of the limiting equation \eqref{equation 1}. This is what we achieve in the present paper.


\subsection{Notation}

For $u,v\in\mathcal D$, we put
$$
\mathcal E [u,v] := \frac{1}{2}\int_{\Sph}\int_{\Sph}\frac{(u(\xi)-u(\eta))(v(\xi)-v(\eta))}{|\xi-\eta|^n}\diff\xi\diff\eta \,.
$$
Moreover, if $u$ is sufficiently regular (for instance, Dini continuous), then we introduce
\begin{equation}
\label{eq:defop}
H u(\xi) := P.V. \int_\Sph \frac{u(\xi) - u(\eta)}{|\xi - \eta|^n} \diff \eta \,.
\end{equation}
Note that in this case, for any $v\in\mathcal D$,
$$
\int_{\Sph} v(\xi) (Hu)(\xi) \diff\xi = \mathcal E[v,u] \,.
$$


\section{Preliminaries}
\label{section prelim}

In this section we prove conformal invariance of equation \eqref{equation 1}. Moreover, we introduce the necessary notation for the conformal maps which our argument relies on, namely inversion and reflection on $\R^n$ and stereographic projection from $\Sph$ to $\R^n$.


\subsection{Conformal invariance}

For a general conformal map $\Phi: X \to Y$ with determinant $J_\Phi(x) := |\det D \Phi(x)|$ and a function $u \in L^2(Y)$, we define the pullback of $u$ under $\Phi$ by
\begin{equation}
\label{definition u Phi} u_\Phi(x) :=  J_\Phi(x)^{1/2} u(\Phi(x)), \qquad x \in X.
\end{equation}
This definition is chosen so that $\|u_\Phi\|_{L^2(X)} = \|u\|_{L^2(Y)}$.

The following lemma shows that equation \eqref{equation 1} is conformally invariant. This is crucial for our approach.

\begin{lemma}
	\label{lemma conf inv}
	Let $u, v \in \mathcal D$ and let $\Phi$ be a conformal map on $\Sph$. Then $u_\Phi, v_\Phi \in \mathcal D$ and we have
	\begin{equation}
	\label{conf transf E} \mathcal E[u_\Phi, v_\Phi] = \mathcal E[u, v] + C_n \int_\Sph u v \ln J_{\Phi^{-1}}^{-1/2} \diff \xi
	\end{equation}
	and, in particular, in the weak sense,
	\begin{equation}
	\label{conf transf H}
	H (u_\Phi)
	= (H u)_\Phi
	+ C_n u_\Phi \ln J_\Phi^\frac 12 \,.
	\end{equation}
	Moreover, if $u$ is a weak solution to \eqref{equation 1}, then so is $u_\Phi$.	
\end{lemma}

To avoid confusion, we emphasize that in the second term on the right side of \eqref{conf transf E}, $\Phi^{-1}$ denotes the inverse of the map $\Phi$, while $J_{\Phi^{-1}}^{-1/2}$ denotes $1/\sqrt{J_{\Phi^{-1}}}$.

\begin{proof}
	\textit{Step 1.}
	For $s > 0$, denote by $\mathcal P_{2s}$ the operator given by
	\begin{align}\label{definition of P}
	\mathcal{P}_{2s}u (\xi) := \frac{\Gamma(\frac{n-2s}{2})}{2^{2s}\pi^{\frac{n}{2}}\Gamma(s)}\int_\Sph\frac{ u(\eta)}{|\xi - \eta|^{n-2s}} \diff \eta.
	\end{align}
	This operator fulfills the conformal invariance property
	\begin{align}\label{property of P}
	\mathcal{P}_{2s} \left(J_\Phi^{\frac{n+2s}{2n}}u\circ{\Phi}\right)=J_\Phi^{\frac{n-2s}{2n}}(\mathcal{P}_{2s}u)\circ{\Phi}.
	\end{align}
	Indeed, this follows from a straightforward change of variables together with the transformation rule
	\begin{equation}
	\label{eq:confdistance}
	J_\Phi(\xi)^{\frac{1}{n}}|\xi-\eta|^2J_\Phi(\eta)^{\frac{1}{n}}=|\Phi(\xi)-\Phi(\eta)|^2,
	\end{equation}
	which holds because $\Phi$ is conformal. 
	
	We next give the action of $\mathcal P_{2s}$ on spherical harmonics. We denote by $(Y_{l,m})$ an orthonormal basis of $L^2(\Sph)$ composed of real spherical harmonics. The index $l$ runs through $\N_0$ and denotes the degree of the spherical harmonic. The index $m$ runs through a certain index set of cardinality depending on $l$ and labels the degeneracy of spherical harmonics of degree $l$.	
	
	By the Funk-Hecke formula (see \cite[Eq.~(17)]{Be1993} and also \cite[Corollary 4.3]{FrLi2012}) we have
	$$
	\mathcal{P}_{2s}Y_{l,m}=\frac{\Gamma(l+n/2-s)}{\Gamma(l+n/2+s)}Y_{l,m}.
	$$
	Expanding $u \in L^2(\Sph)$ in terms of the spherical harmonics,
	\begin{equation}
	\label{sph harm exp}
	u=\sum_{l,m} u_{l,m}Y_{l,m} \qquad \qquad \textrm{with} \qquad u_{l,m}=\int_{\mathbb{S}^n}uY_{l,m} \diff \eta,
	\end{equation}
	by the Funk-Hecke formula (see \cite[Corollary 4.3]{FrLi2012}) we have the representation
	$$
	\mathcal{P}_{2s} u=\sum_{l,m} u_{l,m} \mathcal{P}_{2s}Y_{l,m} =\sum_{l,m} u_{l,m}\frac{\Gamma(l+n/2-s)}{\Gamma(l+n/2+s)}Y_{l,m} \,.
	$$
	In passing, we note that the right side is equal to $A_{2s}^{-1} u$ with the operator $A_{2s}$ from \eqref{eq:opas}.
	
	We denote by $\mathcal F$ the space of functions $u$ on $\Sph$ such that only finitely many coefficients $u_{l,m}$ in \eqref{sph harm exp} are nonzero. All the above computations are justified, in particular, for such functions. Moreover, acting on such functions one has $\lim_{s\to 0} \mathcal P_{2s} = 1 =: \mathcal P_0$.
	
	\textit{Step 2.}
	We now prove \eqref{conf transf H} for $u \in \mathcal F$. For such $u$, we may differentiate the identity \eqref{property of P}
	with respect to $s$ at $s=0$. We note that
	$$
	\mathcal P_0 = 1 \,,
	\qquad
	\dot{\mathcal P}_0 = - \frac{\Gamma(\frac n2)}{2\pi^{\frac n2}} H - \frac{\dot\Gamma(\frac n2)}{\Gamma(\frac n2)} \,.
	$$
	(Here and in the following, we use a dot to denote the derivative with respect to $s$.)	The latter follows from the identity
	\[ H u (\xi)  = \lim_{s \to 0} \int_\Sph \frac{u(\xi) - u(\eta)}{|\xi - \eta|^{n-2s}} \diff \eta  = \frac{2\pi^{\frac n2}}{\Gamma(\frac n2)} \lim_{s\to 0} \frac{1}{2s} \left( \frac{\Gamma(\frac n2 - s)}{\Gamma(\frac n2 +s)} - \mathcal P_{2s}\right) u \]
	by expanding the quotient of gamma functions.
	
	For the left side of \eqref{property of P}, we obtain
	$$
	\frac{d}{ds} \mathcal P_{2s} \left( J_\Phi^\frac{n+2s}{2n} u \circ\Phi \right) = 2\, \dot {\mathcal P}_{2s} \left( J_\Phi^\frac{n+2s}{2n} u \circ\Phi \right) + \frac 1n \mathcal P_{2s} \left( J_\Phi^\frac{n+2s}{2n} \ln J_\Phi \ u \circ\Phi \right)
	$$
	and therefore, at $s=0$,
	$$
	\frac{d}{ds}|_{s=0} \mathcal P_{2s} \left( J_\Phi^\frac{n+2s}{2n} u \circ\Phi \right) = - 2\, \left( \frac{\Gamma(\frac n2)}{2\pi^{\frac n2}} H + \frac{\dot\Gamma(\frac n2)}{\Gamma(\frac n2)} \right) \left( J_\Phi^\frac{1}{2} u \circ\Phi \right) + \frac 1n \left( J_\Phi^\frac{1}{2} \ln J_\Phi \ u \circ\Phi \right).
	$$
	For the right side of \eqref{property of P}, we obtain
	$$
	\frac{d}{ds} J_\Phi^{\frac{n-2s}{2n}} \left( \mathcal P_{2s}u\right)\circ\Phi = -\frac 1n J_\Phi^{\frac{n-2s}{2n}} \ln J_\Phi \left( \mathcal P_{2s}u\right)\circ\Phi
	+ 2\, J_\Phi^{\frac{n-2s}{2n}} \left( \dot{\mathcal P}_{2s}u\right)\circ\Phi
	$$
	and therefore, at $s=0$,
	$$
	\frac{d}{ds}|_{s=0} J_\Phi^{\frac{n-2s}{2n}} \left( \mathcal P_{2s}u\right)\circ\Phi = -\frac 1n J_\Phi^{\frac{1}{2}} \ln J_\Phi \ u \circ\Phi
	- 2\, J_\Phi^{\frac{1}{2}} \left( \left(\frac{\Gamma(\frac n2)}{2\pi^{\frac n2}} H + \frac{\dot\Gamma(\frac n2)}{\Gamma(\frac n2)} \right) u\right)\circ\Phi \,.
	$$
	Combining these two identities and recalling $u_\Phi = J_\Phi^{1/2} u \circ \Phi$, we arrive at equation \eqref{conf transf H}, understood in a pointwise sense.
	
	Now let $v\in\mathcal F$, multiply \eqref{conf transf H} by $v_\Phi$ and integrate over $\Sph$. After a change of variables $\xi \mapsto \Phi(\xi)$ on the right hand side, we obtain the desired identity \eqref{conf transf E} for all $u, v \in \mathcal F$.
	
	\textit{Step 3.}
	We now remove the apriori assumption $u, v \in \mathcal F$. 
	
	From the representation \eqref{eq:energyharmonics} we see that $\mathcal F \subset \mathcal D$ is dense in $\mathcal D$ with respect to the norm $\sqrt{\mathcal E[u,u] + \|u\|_2^2}$.
	
	We need to show that the second term on the right side of \eqref{conf transf E} is harmless. By the classification of conformal maps of $\Sph$, we know that $J_{\Phi^{-1}}(\xi) = (\sqrt{1-|\zeta|^2}/(1-\zeta\cdot\xi))^n$ for some $\zeta\in\R^{n+1}$ with $|\zeta|<1$. Thus,
	$$
	J_{\Phi^{-1}}(\xi) \geq \left( \frac{\sqrt{1-|\zeta|^2}}{1+|\zeta|} \right)^n = \left( \frac{1-|\zeta|}{1+|\zeta|}\right)^{n/2} \,.
	$$
	This implies
	\begin{align*}
	C_n \int_\Sph u^2 \ln J_{\Phi^{-1}}^{-1/2} \diff \xi &\leq \frac{n\,C_n}4 \left( \ln \frac{1+|\zeta|}{1-|\zeta|}\right) \|u\|_2^2 = C_\Phi\, \|u\|_2^2
	\end{align*}
	and
	\begin{equation*}
	\mathcal E[u_\Phi, u_\Phi] \leq \mathcal E[u, u] + C_\Phi\, \|u\|_2^2 \,.  
	\end{equation*}
	
	This bound, together with a standard approximation argument, shows that $u_\Phi \in \mathcal D$ whenever $u \in \mathcal D$ and that \eqref{conf transf E} holds for every $u,v \in \mathcal D$.
	
	\emph{Step 4.}	
	To obtain the statement on conformal invariance of equation \eqref{equation 1}, let $\varphi \in \mathcal D$ and set $v:=\varphi_{\Phi^{-1}}$. Then we compute, using \eqref{conf transf E},
	\begin{align*}
	\mathcal{E}[\varphi, u_\Phi] & =\mathcal{E}[v_\Phi,u_\Phi] = \mathcal{E}[v,u]+ C_n \int_{\mathbb{S}^n} v u \ln J_{\Phi^{-1}}^{-1/2} \diff \xi\\
	& =C_n \int_{\mathbb{S}^n} vu\ln u \diff \xi+C_n \int_{\mathbb{S}^n} vu \ln J_{\Phi^{-1}}^{-1/2}  \diff \xi
	=C_n \int_{\mathbb{S}^n} \varphi u_\Phi \ln u_\Phi \diff \xi.
	\end{align*}
	This finishes the proof. 
\end{proof}


\subsection{Some conformal maps}\label{subsection stereographic}

Our argument will make use of certain one-parameter families of conformal transformations of $\Sph$. It is natural to define these maps on $\R^n$ and then to lift them to the sphere via a stereographic projection. The families in questions are inversions in spheres (with fixed center and varying radii) and reflections in hyperplanes (with fixed normal and varying positions).

Let us set up our notation. On $\R^n$, the inversion about the sphere $\partial B_\lambda(x_0)$ with center $x_0 \in \R^n$ and radius $\lambda >0$ is given by
$$
I_\lx : \R^n \setminus \{x_0\} \to \R^n \setminus \{x_0\} \,,
\qquad
I_{\lambda, x_0}(x) = \frac{\lambda^2 (x-x_0)}{|x-x_0|^2} + x_0 \,.
$$
Similarly, if $H_{\alpha,e}:= \{ x \in \R^n \, : \, x \cdot e > \alpha \}$ denotes the halfspace with normal $e \in \mathbb S^{n-1}$ and position $\alpha \in \R$, the reflection about the hyperplane $\partial H_{\alpha,e}$ is given by
$$
R_{\alpha,e} : \R^n \to \R^n \,,
\qquad
R_{\alpha,e}(x) := x + 2 (\alpha - x \cdot e)e \,.
$$

The inverse stereographic projection $\Ste:  \R^n \to \Sph \setminus \{S\}$, where $S = - e_{n+1}$ denotes the southpole, is given by
\begin{equation}
\label{eq:stereo}
(\Ste(x))_i= \frac{2 x_i}{1 + |x|^2}, \quad i = 1,...,n, \quad (\Ste(x))_{n+1} = \frac{1-|x|^2}{1+|x|^2}.
\end{equation}
Correspondingly, the stereographic projection is given by $\Ste^{-1}: \Sph \setminus \{S\} \to \R^n$,
\[ (\Ste^{-1}(\xi))_i = \frac{\xi_i}{1 + \xi_{n+1}}, \quad i = 1,...,n. \]

Using stereographic projection, we now lift the inversions and reflections to $\Sph$. For any $\lambda>0$ and $\xi_0 \in\Sph\setminus\{S\}$ we set
\begin{equation}
\label{definition Phi lambda xi}
\Phi_\lxi := \Ste \circ I_\lx \circ \Ste^{-1} : \Sph \setminus \{\xi_0, S \} \to \Sph \setminus \{\xi_0,S\}.
\end{equation}
Here and in the following, the relation $\Ste (x_0) = \xi_0$ is understood. The map $\Phi_\lxi$ is conformal, being a composition of conformal maps.
We abbreviate
$$
J_\lxi(\eta) := |\det D \Phi_\lxi (\eta)|
\qquad\text{and}\qquad
\Slxi := \Ste(B_\lambda(x_0)) \,.
$$
Similarly, for any $\alpha\in\R$ and $e\in\mathbb S^{n-1}$ we set
\begin{equation}
\label{definition Psi a e}
\Psi_{\alpha,e} := \Ste \circ R_{\alpha,e} \circ \Ste^{-1} : \, \Sph \setminus \{S\} \to \Sph \setminus \{S\}.
\end{equation}
We abbreviate
$$
J_{\alpha,e}(\eta):= |\det D \Psi_{\alpha,e} (\eta)| \,.
$$

To motivate the following lemma, we recall that when applying the method of moving planes in $\R^n$ in integral form, the following inequality is fundamental,
\begin{equation}
\label{difference is negative on Rn}
|x - y| < |x- R_{0,e}(y)| \qquad \text{ for all } x, y \in H_{0,e} \,.
\end{equation}
The following lemma yields the corresponding inequalities for the lifted maps $\Phi_\lxi$ and $\Psi_{\alpha,e}$ on $\Sph$.

\begin{lemma}
\label{lemma difference is negative}
\begin{enumerate}
\item[(i)]
Let $\lambda > 0$ and let $\xi_0 \in \Sph \setminus \{S\}$. Then
\begin{equation}
\label{difference is negative on Sph}
\frac{|J_\lxi(\eta)|^{1/2}}{|\xi - \Phi_\lxi(\eta)|^n} - \frac{1}{|\xi - \eta|^n} < 0
\qquad\text{for all $\xi, \eta \in \Slxi$ with $\xi \neq \eta$}.
\end{equation}

\item[(ii)]
Let $\alpha \in \R$ and let $e \in \mathbb S^{n-1}$. Then
\[ \frac{|J_{\alpha,e}(\eta)|^{1/2}}{|\xi - \Psi_{\alpha,e}(\eta)|^n} - \frac{1}{|\xi - \eta|^n} < 0 
\qquad\text{for all $\xi, \eta \in \Ste(H_{\alpha,e})$ with $\xi \neq \eta$}.
\]
\end{enumerate}
\end{lemma}

\begin{proof}
Inequality \eqref{difference is negative on Sph} is equivalent to the inequality
\begin{equation}
\label{difference on Sph ver 1}
 |\xi - \eta| < |\xi - \Phi_\lxi (\eta)| |J_\lxi(\eta)|^{-1/2n} \qquad \text{ for all } \xi, \eta \in \Sigma_\lxi.
\end{equation}

Our strategy is to deduce this inequality from \eqref{difference is negative on Rn} (with $e=e_n$, say). We observe that there is a conformal map $\mathcal B_\lx : \R^n \setminus \{x_0, x_0 - \lambda e_n\} \to \R^n \setminus \{x_0, x_0 - \lambda e_n\}$, which maps the ball $B_\lambda(x_0)$ to the half-space $H_{0,e_n}$ and which is such that $I_\lx = \mathcal B_\lx^{-1} \circ R_{0,e_n} \circ \mathcal B_\lx$. (See \cite[Section 2.2]{FrLi} for details; the map $\mathcal B_\lx$ is the map $\mathcal B$ given there, composed with a dilation and a translation which map $B_\lambda(x_0)$ to $B_1(0)$.)

Therefore, letting $\mathcal T = \mathcal T_\lx = \mathcal B_\lx \circ \mathcal S^{-1}$, we can write $\Phi_\lxi = \mathcal T^{-1} \circ R_{0,e_n} \circ \mathcal T$. If we use the formula $|\mathcal T(\xi) - \mathcal T(\eta)| = |\det D \mathcal T(\xi)|^{1/2n} |\det D \mathcal T(\eta)|^{1/2n} |\xi - \eta|$ (valid for any conformal map, see \eqref{eq:confdistance}) and the chain rule for the determinantal factors, we get that \eqref{difference on Sph ver 1} is equivalent to
\[ |\mathcal T(\xi) - \mathcal T(\eta)| < |\mathcal T(\xi) - \mathcal T( \Phi_\lx (\eta))| |\det D(\mathcal T \circ \Phi_\lx \circ \mathcal T^{-1})(\mathcal T(\eta))|^{\frac{1}{2n}} \ \ \text{for all } \xi, \eta \in \Sigma_\lxi. \]
Setting $x = \mathcal T (\xi)$ and $y = \mathcal T (\eta)$ and observing that $\mathcal T(\Sigma_\lxi) = H_{0,e_n}$, this simplifies to
\[ |x - y| < |x- R_{0,e_n}(y)| |\det D R_{0,e_n}(y)|^{1/2n} \qquad \text{ for all } x, y \in H_{0,e_n},  \]
which is just \eqref{difference is negative on Rn} because $|\det D R_{0,e_n}(y)| =1$ for all $y \in \R^n$.

The proof of (ii) is similar, but simpler (instead of $\mathcal T$ one can take simply $\Ste^{-1}$) and we omit it.
\end{proof}


\section{Maximum principles for antisymmetric functions}
\label{section max principles}

This section serves as a preparation for the moving spheres argument carried out in Section \ref{section moving spheres} below. Here we will derive two maximum principles which will be the technical heart of the moving spheres method. An important point, which makes them well suited for the moving spheres application, is that both Lemmas \ref{lemma max principle narrow region} and \ref{lemma strong max principle} only hold for antisymmetric functions. Here, with the notation introduced in Section \ref{section prelim}, we call a function $w$ on $\Sph$ \emph{antisymmetric} (with respect to the conformal maps $\Phi_\lxi$ resp. $\Psi_{\alpha,e}$) if
\[ w(\eta) = -w_{\Phi_\lxi}(\eta) \qquad \text{ for a.e. } \quad  \eta \in \Slxi, \]
respectively,
\[ w(\eta) = -w_{\Psi_{\alpha,e}}(\eta) \qquad \text{ for a.e. } \quad  \eta \in \Ste(H_{\alpha,e}). \]
The use of maximum principles is fundamental in the method of moving planes and the role of antisymmetry in these maximum principles becomes particularly important when applied to nonlocal equations. Antisymmetric maximum principles are implicit, among others, in \cite{ChLiOu,Li04,MaZh,FrLe,JiXi} and were made explicit in \cite{JaWe,ChLiLi,ChLiZh}. A particular feature of our result is that we deal with weak solutions without further regularity assumptions and with very weak conditions on the potential. This makes, for instance, the proof of the strong maximum principle from Lemma \ref{lemma strong max principle} below considerably more involved than its counterpart in \cite{ChLiLi}.

For clarity of exposition, we will state and prove the lemmas in this section only for functions antisymmetric with respect to $\Phi_\lxi$. We leave it to the reader to check that their statements and proofs remain valid when $\Phi_\lxi$ and $\Slxi$ are replaced by $\Psi_{\alpha,e}$ and $\Ste(H_{\alpha,e})$.

The first lemma states a maximum principle which is valid for sets of sufficiently small volume. We recall that the constant $C_n$ was defined in \eqref{eq:becknerconst}.

\begin{lemma}[Small volume maximum principle]
\label{lemma max principle narrow region}

Let $\lambda >0$ and $\xi_0 \in \Sph \setminus \{S\}$, let $\Omega \subset \Slxi$ be measurable and let $V: \Omega \to \R$ be a measurable function with
$$
\int_\Omega e^{2 V_-/C_n} < |\Sph| \,.
$$
If $w \in \mathcal D$ is antisymmetric with respect to $\Slxi$ and satisfies
\begin{equation}
\label{assumptions on w narrow region mp 1} \mathcal E[\varphi,w] + \int_\Omega \varphi V w \geq 0 \,\, \text{ for any } 0\leq\varphi\in{\mathcal{D}} \text{ with } \varphi=0 \text{ on } \Omega^c
\end{equation}
and
\begin{equation}
\label{assumptions on w narrow region mp 2}
w \geq 0 \,\, \text{ a.e.~on } \Slxi \setminus \Omega,
\end{equation}
then $w \geq 0$ a.e.~on $\Omega$.
\end{lemma}

We prove Lemma \ref{lemma max principle narrow region} in Section \ref{subsection proof of narrow region} below.

The second lemma gives a strong maximum principle.

\begin{lemma}[Strong maximum principle]
\label{lemma strong max principle}
Let $\lambda > 0$ and $\xi_0 \in \Sph \setminus \{S\}$ and let $V:\Slxi\to\R$ be a measurable function. If $w\in\mathcal D$ is antisymmetric with respect to $\Slxi$, satisfies $V_+ \min\{w,1\} \in L^1_{\rm loc}(\Slxi)$, as well as
\begin{equation}
\label{assumptions on w strong mp 1}
\mathcal E[\varphi, w] + \int_{\Sigma_\lxi} \varphi V w \geq 0
\qquad\text{for all}\ 0\leq\varphi\in\mathcal D \text{ with } \varphi=0 \text{ on } (\Slxi)^c 
\end{equation}
and
\begin{equation}
\label{assumptions on w strong mp 2}
w \geq 0 \,\, \text{ a.e.~on } \Slxi,
\end{equation}
then either $w \equiv 0$ on $\Sph$ or $w > 0$ a.e.~on $\Slxi$.
\end{lemma}

We prove Lemma \ref{lemma strong max principle} in Section \ref{subsection proof of strong max} below.


\subsection{Proof of Lemma \ref{lemma max principle narrow region}}
\label{subsection proof of narrow region}

To prove Lemma \ref{lemma max principle narrow region}, we use the following inequality. It is well known, but we include a proof for the sake of completeness.

\begin{lemma}
\label{lemma gibbs}
Let $f$ and $g$ be measurable functions on a measured space $(X, \mu)$ such that $f\geq 0$ and $\int_X f\diff\mu =1$. Assume that $\int_X fg_-\diff\mu<\infty$. Then
$$
\int_X fg \diff \mu \leq \int_X f\ln f \diff \mu + \ln \int_X e^g \diff \mu \,.
$$
\end{lemma}

\begin{proof}
From the elementary inequality $e^{\alpha}\geq 1+ \alpha$ for all $\alpha\in\R$, we obtain
$$
e^{g-\ln f - \int_X f(g-\ln f) \diff \mu} \geq 1+ g-\ln f - \int_X f(g-\ln f) \diff \mu\,.
$$
Multiplying this inequality by $f$ and integrating we obtain
$$
\int_X e^{g - \int_X f(g-\ln f) \diff \mu} \diff \mu \geq 1 \,,
$$
which is the claimed inequality.
\end{proof}

\begin{lemma}\label{cutting}
Let $\lambda >0$ and $\xi_0 \in \Sph \setminus \{S\}$ and let $w \in \mathcal D$ be antisymmetric with respect to $\Slxi$. Then $v := \1_{\Slxi}\, w_{-}\in\mathcal D$ and
$$
\mathcal E[v,v]\leq-\mathcal E[v,w] \,.
$$
\end{lemma}

\begin{proof}
	For $\epsilon>0$ set
	\begin{align*}
	E_\epsilon & := \{(\xi,\eta)\in\Sigma_\lxi\times\Sigma_\lxi :\, |\xi-\eta|<\epsilon\} \cup \{(\xi,\eta)\in\Sigma_\lxi\times\Sigma_\lxi^c :\, |\xi-\Phi_\lxi(\eta)|<\epsilon\} \\
	& \qquad \cup \{(\xi,\eta)\in\Sigma_\lxi^c \times\Sigma_\lxi :\, |\Phi_\lxi(\xi)-\eta|<\epsilon\} \\
	& \qquad \cup \{(\xi,\eta)\in\Sigma_\lxi^c\times\Sigma_\lxi^c :\, |\Phi_\lxi(\xi)-\Phi_\lxi(\eta)|<\epsilon\} \,,
	\end{align*}
	$$
	k_\epsilon(\xi,\eta) := \1_{E_\epsilon^c}(\xi,\eta)\, |\xi-\eta|^{-n}
	$$
	and
	$$
	\mathcal E_\epsilon[f,g] := \frac{1}{2}\int_{\Sph}\int_{\Sph}(u(\xi)-u(\eta))(v(\xi)-v(\eta))
	k_\epsilon(\xi,\eta) \diff\xi\diff\eta \,.
	$$
	Since $k_\epsilon$ is bounded, both $\mathcal E_\epsilon[v,v]$ and $\mathcal E_\epsilon[w,v]$ are finite and we have
	\begin{align*}
	& \mathcal E_\epsilon[v,v]+ \mathcal E_\epsilon[v,w]\\
	& =\frac{1}{2}\int_{\Sph}\int_{\Sph}\left( |v(\xi)-v(\eta)|^2+[v(\xi)-v(\eta)][w(\xi)-w(\eta)] \right) 
	k_\epsilon(\xi,\eta) \diff\xi\diff\eta\\
	& =\int_{\Sph}\int_{\Sph} \left( v(\xi)(v(\xi)+w(\xi)) - v(\xi)(v(\eta)+w(\eta)) \right) 
	k_\epsilon(\xi,\eta) \diff\xi\diff\eta\\
	& =-
	\int_{\Sph}\int_{\Sph} v(\xi)(v(\eta)+w(\eta))
	k_\epsilon(\xi,\eta) \diff\xi\diff\eta \,,
	\end{align*}
	where we used the fact that $v(\xi)(v(\xi)+w(\xi))=0$ on $\Sph$.
	
	By a change of variables and the antisymmetry of $w$, we have
	\begin{align*}
	& -\int_{\Sph}\int_{\Sph} v(\xi)(v(\eta)+w(\eta)) k_\epsilon(\xi,\eta) \diff\xi\diff\eta\\
	& =-\int_{\Sigma_{\lambda,\xi_{0}}}\int_{\Sigma_{\lambda,\xi_{0}}}
	w(\xi)_- w(\eta)_+ k_\epsilon(\xi,\eta) \diff\xi\diff\eta
	+ \int_{\Sigma_{\lambda,\xi_{0}}^c}\int_{\Sigma_{\lambda,\xi_{0}}} w(\xi)_- w(\eta)_- k_\epsilon(\xi,\eta) \diff\xi\diff\eta \\
	& \quad - \int_{\Sigma_{\lambda,\xi_{0}}^c}\int_{\Sigma_{\lambda,\xi_{0}}} w(\xi)_- w(\eta)_+ k_\epsilon(\xi,\eta) \diff\xi\diff\eta \\
	& = - \int_{\Sigma_{\lambda,\xi_{0}}}\int_{\Sigma_{\lambda,\xi_{0}}}w(\xi)_-w(\eta)_+ \left( k_\epsilon(\xi,\eta) - J_{\lambda,\xi_{0}}(\eta)^{1/2} k_\epsilon(\xi,\Phi_\lxi(\eta)) \right) \diff\xi\diff\eta  \\
	& \quad
	- \int_{\Sigma_{\lambda,\xi_{0}}^c}\int_{\Sigma_{\lambda,\xi_{0}}} w(\xi)_- w(\eta)_+ k_\epsilon(\xi,\eta) \diff\xi\diff\eta \,.
	\end{align*}
	The second double integral on the right side is clearly nonnegative. Moreover, it follows from Lemma \ref{lemma difference is negative} and the choice of $E_\epsilon$ (more precisely, the fact that $(\xi, \Phi(\eta)) \in E_\epsilon$ implies $(\xi, \eta) \in E_\epsilon$) that
	$$
	k_\epsilon(\xi,\eta) - J_{\lambda,\xi_{0}}(\eta)^{1/2} k_\epsilon(\xi,\Phi_\lxi(\eta)) \geq 0
	\qquad\text{for all}\ \xi,\eta\in\Sigma_\lxi \,.
	$$
	Therefore also the first double integral on the right side is nonnegative and we conclude that
	\begin{equation}
	\label{eq:ineqregularized}
	\mathcal E_\epsilon[v,v] + \mathcal E_\epsilon[v,w] \leq 0 \,.
	\end{equation}
	
	By the Schwarz inequality, we have $-\mathcal E_\epsilon[v,w]\leq \sqrt{\mathcal E_\epsilon[v,v]} \sqrt{\mathcal E_\epsilon[w,w]}$ and therefore the previous inequality implies that
	$$
	\mathcal E_\epsilon[v,v] \leq \mathcal E_\epsilon[w,w] \,.
	$$
	Since $\mathcal E_\epsilon[w,w]\leq \mathcal E[w,w]<\infty$, we can let $\epsilon\to 0$ and use monotone convergence to deduce that $\mathcal E[v,v]<\infty$, that is, $v\in\mathcal D$. With this information we can return to \eqref{eq:ineqregularized} and let $\epsilon\to 0$ to obtain the inequality in the lemma.	
\end{proof}

Now we can give the proof of the small volume maximum principle.

\begin{proof}[Proof of Lemma \ref{lemma max principle narrow region}]
Let $\lambda$, $\xi_0$, $w$ and $\Omega$ be as in the assumptions and denote $v=\1_{\Omega} w_{-}$. Assumption \eqref{assumptions on w narrow region mp 2} implies that $v=\1_{\Sigma_\lxi}w_-$ and therefore, by Lemma \ref{cutting}, $v\in\mathcal D$ and $\mathcal E[v,v]\leq \mathcal E[v,w]$. Combining this with assumption \eqref{assumptions on w narrow region mp 1} (with $\varphi=v$), we obtain
\begin{equation}
\label{max princ eq 1} \mathcal E[v,v]  \leq -\mathcal E[v,w] \leq \int_\Omega v Vw =-\int_{\Omega} V w_{-}^2 \leq \int_{\Omega} V_- w_{-}^2 \,.
\end{equation}
On the other hand, by Beckner's inequality \eqref{beckner log sob},
\begin{align}
& \mathcal E[v,v] =\frac{1}{2}\int_{\Sph}\int_{\Sph}\frac{|v(\xi)-v(\eta)|^2}{|\xi-\eta|^n}\diff\xi\diff\eta \geq \frac{C_n}2 \int_{\Omega} w_-^2 \ln \frac{w_-^2 |\Sph|}{\|w_- 1_\Omega \|_2^2}\diff \eta.  \label{max princ eq 2}
\end{align}
We now argue by contradiction and assume that $v \nequiv 0$. Then we may define $\rho = \|w_- \1_\Omega\|_2^{-2} \1_{\Omega} w_-^2$
and rewrite inequalities \eqref{max princ eq 1} and \eqref{max princ eq 2} as
\begin{equation}
\label{inequality with rho} \int_\Omega \rho \ln \rho \diff \eta + \ln |\Sph| \leq \frac{2}{C_n} \int_\Omega V_- \rho \diff \eta.
\end{equation}
Notice that $\rho \geq 0$ and $\int_\Sph \rho = 1$. Therefore we may apply Lemma \ref{lemma gibbs} with $f = \rho$, $g= \frac{2}{C_n} V_-$ and $X = \Omega$ to \eqref{inequality with rho} and deduce that
\[ |\Sph| \leq \int_\Omega e^\frac{2 V_-}{C_n} \diff \eta \, . \]
This contradicts the assumption of the lemma and therefore we conclude that $v\equiv 0$, that is, $w\geq 0$ a.~e.~on $\Omega$.
\end{proof}


\subsection{A general form of the strong maximum principle}

We deduce Lemma~\ref{lemma strong max principle} from the following strong maximum principle, which holds for a general interaction kernel $k$ and arbitrary, not necessarily antisymmetric, functions $w$. 

More precisely, let $(X,d,\mu)$ be a metric measure space and suppose that the kernel $k: X \times X \to (0,\infty)\cup\{\infty\}$ is such that
\begin{equation}
\label{bound k} 
\iint_{K \times X} k(x,y)d(x,y)^2 \diff\mu(x) \diff\mu(y) <\infty
\qquad\text{for any compact}\ K \subset X \,. 
\end{equation}
For simplicity of notation, we will also assume that $k(x,y)=k(y,x)$ for all $x,y\in X$, although this is not really necessary. For a measurable function $u$ on $X$, let
$$
\mathcal{I}[w] := \frac12 \iint_{X \times X} k(x,y) (w(x)-w(y))^2 \diff \mu(x) \diff \mu(y)
$$
and set
\[ D(\mathcal I) := \{ w : X \to \R :\  w \ \text{measurable} \,,\ \mathcal{I}[w] < \infty \}. \]
Moreover, for $v,w\in D(\mathcal I)$, let
\[ \mathcal I[v,w] := \frac12 \iint_{X \times X} k(x,y) (v(x)-v(y)) (w(x) - w(y)) \diff \mu(x) \diff \mu(y) \,. \]
With this notation, we can state the following general strong maximum principle.

\begin{proposition}
\label{lemma strong max general}
Let $k$ satisfy \eqref{bound k} and let $U$ be a measurable function on $X$. Assume that $0\leq v\in D(\mathcal I)$ satisfies $U_+ \min\{v,1\} \in L^1_{\rm loc}(X)$ and
\begin{equation}
\label{mp general ineq}
 \mathcal I [\varphi,v] + \int_X \varphi\, U v\diff\mu \geq 0 
\qquad\text{for all}\ 0 \leq\varphi \in D(\mathcal I) \ \text{with compact support} \,.
 \end{equation}
Then either $v \equiv 0$ or $v > 0$ a.e.~on $X$.
\end{proposition}

In the proof of Lemma \ref{lemma strong max principle}, we will use this proposition in a setting where in fact $U_+v\in L^1_{\rm loc}(X)$. While this may at first look less natural than the assumption $U_+\in L^1_{\rm loc}(X)$, the difference is crucial in our application, where $U$ depends in a nonlinear fashion on $v$.

We begin with a technical lemma about the form domain $D(\mathcal I)$.

\begin{lemma}\label{operations}
	Let $k$ satisfy \eqref{bound k} and let $w\in D(\mathcal I)$.
	\begin{enumerate}
		\item [(a)] If $w\geq 0$, then $(w+\epsilon)^{-1}\in D(\mathcal I)$ for all $\epsilon>0$.
		\item[(b)] If $w$ is bounded and $\zeta$ is Lipschitz function on $X$ with compact support, then $\zeta w\in D(\mathcal I)$.
	\end{enumerate}
\end{lemma}

\begin{proof}[Proof of Lemma \ref{operations}]
	To prove (a), we write
$$
\left( \frac{1}{w(x)+\epsilon} - \frac{1}{w(y)+\epsilon} \right)^2 = \frac{(w(x)-w(y))^2}{(w(x)+\epsilon)^2(w(y)+\epsilon)^2} \leq \frac{1}{\epsilon^4} (w(x)-w(y))^2 \,.
$$
Thus,
$$
\mathcal I[(w+\epsilon)^{-1}] \leq \epsilon^{-4} \ \mathcal I[w] \,.
$$

To prove (b), we first note that $\mathcal I[\zeta]<\infty$. Indeed, if $K:=\supp\zeta$ and $L$ is the Lipschitz constant of $\zeta$, then
\begin{align*}
\mathcal I[\zeta] & \leq \int_K \int_X \left(\zeta(x)-\zeta(y)\right)^2 k(x,y)\diff\mu(y)\diff\mu(x) \\
& \leq L^2 \int_K \int_X d(x,y)^2 k(x,y)\diff\mu(y)\diff\mu(x) <\infty
\end{align*}
by \eqref{bound k}. Now we bound
$$
\left| \zeta(x)w(x)-\zeta(y)w(y) \right| \leq \|\zeta\|_\infty \left|w(x)-w(y)\right| + \|w\|_\infty \left|\zeta(x)-\zeta(y)\right| 
$$
and obtain
$$
\mathcal I[\zeta w] \leq \|\zeta\|_\infty^2\, \mathcal I[w] + \|w\|_\infty^2\, \mathcal I[\zeta] + 2 \|\zeta\|_\infty \|w\|_\infty \sqrt{I[w]\,\mathcal I[\zeta]} \,.
$$
This proves the lemma.
\end{proof}

\begin{proof}
[Proof of Proposition \ref{lemma strong max general}]
Let $\zeta$ be a Lipschitz function on $X$ with compact support. By the first part of Lemma \ref{operations}, the function $(v+\epsilon)^{-1}$ belongs to $D(\mathcal I)$ and is bounded, so by the second part, with $\zeta$ replaced by $\zeta^2$, the function $\varphi = \zeta^2/(v+\epsilon)$ belongs to $D(\mathcal I)$.

We write
\begin{align*}
& \left(\varphi(x)-\varphi(y)\right) \left( v(x)-v(y)\right) \\
& = \frac{-(\zeta(x)v(y)-\zeta(y)v(x))^2 + v(x)v(y)(\zeta(x)-\zeta(y))^2 + \epsilon(v(x)-v(y))(\zeta(x)^2-\zeta(y)^2)}{(v(x)+\epsilon)(v(y)+\epsilon)}.
\end{align*}
Using also
$$
\int_X \varphi U v \leq \int_X U_+ \zeta^2 \frac{v}{v+\epsilon} \,,
$$
we obtain from \eqref{mp general ineq} that
\begin{align}
\label{eq:mpgenproof}
& \frac12 \iint_{X\times X} \frac{(\zeta(x)v(y)-\zeta(y)v(x))^2}{(v(x)+\epsilon)(v(y)+\epsilon)} k(x,y)\diff\mu(x)\diff\mu(y)
 \leq \sum_{k=1}^3 I_k(\epsilon)
\end{align}
with
\begin{align*}
I_1(\epsilon) & = \int_X U_+ \zeta^2 \frac{v}{v+\epsilon} \,,\\
I_2(\epsilon) & = \frac12 \iint_{X\times X} \frac{v(x)v(y)(\zeta(x)-\zeta(y))^2}{(v(x)+\epsilon)(v(y)+\epsilon)} k(x,y)\diff\mu(x)\diff\mu(y) \,,\\
I_3(\epsilon) & = \frac12 \iint_{X\times X} \frac{\epsilon(v(x)-v(y))(\zeta(x)^2-\zeta(y)^2)}{(v(x)+\epsilon)(v(y)+\epsilon)} k(x,y)\diff\mu(x)\diff\mu(y) \,.
\end{align*}

We bound the left side of \eqref{eq:mpgenproof} from below. Setting $Z:=\{v=0\}$, we have
\begin{align*}
& \frac12 \iint_{X\times X} \frac{(\zeta(x)v(y)-\zeta(y)v(x))^2}{(v(x)+\epsilon)(v(y)+\epsilon)} k(x,y)\diff\mu(x)\diff\mu(y) \\
& \geq \iint_{Z^c\times Z} \frac{(\zeta(x)v(y)-\zeta(y)v(x))^2}{(v(x)+\epsilon)(v(y)+\epsilon)} k(x,y)\diff\mu(x)\diff\mu(y) \\
& = \epsilon^{-1} \int_{Z^c} \frac{v(x)^2}{v(x)+\epsilon} \kappa(x)\diff\mu(x)
\end{align*}
with
$$
\kappa(x) := \int_{Z} \zeta(y)^2 k(x,y)\diff\mu(y) \,.
$$

By dominated convergence, we have
$$
\lim_{\epsilon\to 0} \int_{Z^c} \frac{v(x)^2}{v(x)+\epsilon} \kappa(x)\diff\mu(x) = \int_{Z^c} v(x) \kappa(x)\diff\mu(x) \,,
$$
and therefore, by \eqref{eq:mpgenproof},
\begin{equation}
\label{eq:mpgenproof2}
\int_{Z^c} v(x) \kappa(x)\diff\mu(x) \leq  \liminf_{\epsilon\to 0} \sum_{k=1}^3  \epsilon I_k(\epsilon) \,.
\end{equation}

Let us show that $\lim_{\epsilon\to 0} \epsilon I_k(\epsilon) =0$ for $k=1,2,3$. We write the integrand of $\epsilon I_1(\epsilon)$ as
$$
U_+ \zeta^2 \min\{v,1\} \left( \frac{\epsilon}{v+\epsilon} \1_{\{0<v<1\}} + \epsilon \frac{v}{v+\epsilon} \1_{\{v\geq 1\}} \right).
$$
By assumption, the product in front of the parentheses is integrable. The factor in parentheses is $\leq 1$ if $\epsilon\leq1$ and tends to zero pointwise. Therefore, by dominated convergence, $\epsilon I_1(\epsilon)\to 0$. Moreover, we can simply bound $I_2(\epsilon) \leq \mathcal I[\zeta]$, which is finite as shown in the proof of Lemma \ref{operations}. Thus, $\epsilon I_2(\epsilon)\to 0$. Finally, the integrand of $\epsilon I_3(\epsilon)$ is bounded, in absolute value, by 
$$
2\|\zeta\|_\infty |v(x)-v(y)| |\zeta(x)-\zeta(y)| k(x,y) \,,
$$
which is integrable. Moreover, this integrand tends pointwise to zero. Thus, by dominated convergence, $\epsilon I_3(\epsilon)\to 0$.

Returning to \eqref{eq:mpgenproof2}, we infer that
\begin{equation}
\label{eq:mpgenproof3}
\int_{Z^c} v(x) \kappa(x)\diff\mu(x) = 0 \,.
\end{equation}

Assume now that $Z$ has positive measure. Then we can choose the function $\zeta$ in such a way that $\zeta^2\1_Z$ is not identically zero. Then, since $k>0$ on $X\times X$, we have $\kappa>0$ on $X$. Thus, by \eqref{eq:mpgenproof3}, $|Z^c|=0$, that is, $v\equiv 0$. This completes the proof.
\end{proof}

\begin{remark}
	There is also a global version of Proposition \ref{lemma strong max general}. Namely, the same conclusion holds without an underlying metric and without assumption \eqref{bound k}, provided one has the global integrability $U_+\min\{v, 1\}\in L^1(X)$ and the compact support condition in \eqref{mp general ineq} is dropped. This follows by the same proof with $\zeta\equiv 1$.	
\end{remark}


\subsection{Proof of Lemma \ref{lemma strong max principle}}
\label{subsection proof of strong max}

It remains to reduce Lemma \ref{lemma strong max principle} to the general maximum principle from Proposition \ref{lemma strong max general} from the previous subsection. To do so, we use antisymmetry of $w$ to express the quadratic form $\mathcal E[\varphi,w]$ as a double integral over the region $\Slxi$, plus a multiplicative term. We drop in the following the subscript from $\Slxi$ and $\Phi_\lxi$ to ease notation. Moreover, we set
\[l(\xi, \eta) := \frac{1}{|\xi - \eta|^n}- \frac{J^{1/2}_\Phi(\eta)}{|\xi- \Phi(\eta)|^n}. \]
Notice that $l(\xi, \eta) > 0$ for every $\xi, \eta \in \Sigma$ by Lemma \ref{lemma difference is negative}. Moreover, by \eqref{eq:confdistance}, we have $l(\xi,\eta)=l(\eta,\xi)$ for all $\xi,\eta\in\Sigma$. For functions $u,v$ on $\Sigma$, we define the quadratic form
\begin{equation}
\label{E tilde} \tilde{\mathcal E}[u,v] := \frac12 \iint_{\Sigma \times \Sigma} l(\xi, \eta) (u(\xi) - u(\eta)) (v(\xi)-v(\eta)) \diff \xi \diff \eta
\end{equation}
on the domain 
\[ \tilde{\mathcal D} := \{ u \in L^2(\Sigma) \, : \, \tilde{\mathcal{E}}[u,u] < \infty \}. \]

\begin{lemma}
\label{lemma E on A}
Let $w \in \mathcal D$ be antisymmetric with respect to $\Phi$ and let $\varphi \in \mathcal D$ with $\varphi=0$ on $\Sigma^c$. Then
\[ \mathcal E[\varphi,w] = \tilde{\mathcal E}[ \varphi,w|_\Sigma] + \int_\Sigma \varphi(\xi) \tilde{V}(\xi) w(\xi) \diff \xi, \]
with
\[ \tilde{V}(\xi) = \int_\Sigma  \frac{J_\Phi(\eta)^{1/2}}{|\xi- \Phi(\eta)|^n}(1+ J_\Phi(\eta)^{1/2}) \diff \eta. \]
\end{lemma}

\begin{proof}
We write
\begin{align}
\label{eq:eonaproof}
\mathcal E[w,\varphi] &= \frac12 \iint_{\Sigma \times \Sigma} \frac{(\varphi(\xi) - \varphi(\eta)) (w(\xi)-w(\eta))}{|\xi - \eta|^{n}} \diff \xi \diff \eta 
+ \iint_{\Sigma \times \Sigma^c} \frac{\varphi(\xi) (w(\xi)-w(\eta))}{|\xi - \eta|^{n}} \diff \xi \diff \eta \,.
\end{align}
The second integral on the right side is a sum of two terms, corresponding to $w(\xi)$ and $w(\eta)$, respectively. Since
$$
\int_{\Sigma^c} \frac{\diff\eta}{|\xi - \eta|^{n}} = \int_{\Sigma} \frac{J_\Phi(\eta)\diff\eta}{|\xi - \Phi(\eta)|^{n}} \,,
$$
the first term becomes
$$
\iint_{\Sigma \times \Sigma^c} \frac{\varphi(\xi) w(\xi)}{|\xi - \eta|^{n}} \diff \xi \diff \eta
= \int_\Sigma \varphi(\xi)w(\xi) \left( \int_{\Sigma} \frac{J_\Phi(\eta)\diff\eta}{|\xi - \Phi(\eta)|^{n}} \right)\diff\xi \,.
$$
This is one contribution of the $\tilde V$ term.

Let us discuss the second contribution coming from the second integral on the right side of \eqref{eq:eonaproof}. By antisymmetry and a change of variables, we have
$$
\int_{\Sigma^c} \frac{w(\eta)}{|\xi-\eta|^n}\diff\eta = - \int_\Sigma
\frac{J_\Phi(\eta)^{1/2}}{|\xi- \Phi(\eta)|^n} w(\eta) \diff \eta \,,
$$
and therefore, by symmetry,
\begin{align*}
\iint_{\Sigma \times \Sigma^c} \frac{\varphi(\xi) w(\eta)}{|\xi - \eta|^{n}} \diff \xi \diff \eta
& = - \frac12 \iint_{\Sigma \times \Sigma} \frac{J_\Phi(\eta)^{1/2}}{|\xi- \Phi(\eta)|^n} \left( \varphi(\xi) w(\eta)+\varphi(\eta)w(\xi) \right) \diff \xi \diff \eta \\
& = \frac12 \iint_{\Sigma \times \Sigma} \frac{J_\Phi(\eta)^{1/2}}{|\xi- \Phi(\eta)|^n} (\varphi(\xi)-\varphi(\eta))(w(\xi)-w(\eta)) \diff \xi \diff \eta \\
& \quad - \frac12 \iint_{\Sigma \times \Sigma} \frac{J_\Phi(\eta)^{1/2}}{|\xi- \Phi(\eta)|^n} \left( \varphi(\xi) w(\xi) + \varphi(\eta) w(\eta) \right) \diff \xi \diff \eta \,.
\end{align*}
In this expression, the first double integral combines with the first double integral on the right side of \eqref{eq:eonaproof} to give the term $\tilde{\mathcal E}[ \varphi,w|_\Sigma]$ in the lemma. Moreover, by symmetry the second double integral equals
$$
\frac12 \iint_{\Sigma \times \Sigma} \frac{J_\Phi(\eta)^{1/2}}{|\xi- \Phi(\eta)|^n} \left( \varphi(\xi) w(\xi) + \varphi(\eta) w(\eta) \right) \diff \xi \diff \eta
= \int_\Sigma \varphi(\xi)w(\xi) \left( \int_{\Sigma} \frac{J_\Phi(\eta)^{1/2}\diff\eta}{|\xi - \Phi(\eta)|^{n}} \right)\diff\xi \,,
$$
which is the remaining contribution to the $\tilde V$ term. Collecting all terms we arrive at the formula in the lemma.
\end{proof}

\begin{proof}
[Proof of Lemma \ref{lemma strong max principle}]
We are going to apply Lemma \ref{lemma strong max general} with $\mathcal I = \tilde{\mathcal E}$, $k = l$, $X = \Slxi$ and $U = V + \tilde{V}$. Let us check that the assumptions of Lemma \ref{lemma strong max general} are satisfied.

By Lemma \ref{lemma E on A}, we have, for any $0\leq\varphi \in \mathcal D$ with $\varphi=0$ on $(\Slxi)^c$,
\begin{equation}
\label{ineq tilde E}
\tilde{\mathcal E}[ \varphi,w] + \int_\Slxi \varphi(\xi) U(\xi) w(\xi) \diff \xi \geq 0 \,.
\end{equation}
Next, we observe that for any compact subset $C \subset \Slxi$, there is $M > 0$ such that we have the uniform bound
\begin{equation}
\label{unif bd}
|\xi- \Phi(\eta)|^{-n}  \leq M \quad \text{ for } \quad \xi \in C, \eta \in \Slxi,
\end{equation}
This has two consequences. First, if $\varphi$ is compactly supported on $\Slxi$, it is easy to deduce from \eqref{unif bd} that $\varphi \in \mathcal D$ if and only if $\varphi \in \tilde{\mathcal D}$. Therefore, \eqref{ineq tilde E} holds for all compactly supported $\varphi \in \tilde{\mathcal D}$.

Second, it follows from \eqref{unif bd} that $\tilde{V}$ is bounded on $C$ and hence $\tilde V \min\{v,1\}\in L^1(C)$. Since, moreover, $V_+ \min\{w,1\} \in L^1(C)$ by assumption, we have $U_+ \min\{w,1\} \in L^1(C)$ for every compact $C \subset \Slxi$.

Thus, all the assumptions of Lemma \ref{lemma strong max general} are satisfied and we conclude by that lemma.
\end{proof}


\section{Symmetry by the method of moving spheres}
\label{section moving spheres}

In this section, we prove a symmetry result for solutions of \eqref{equation 1}. We will deduce this by the method of moving spheres using the preliminaries introduced so far, in particular the maximum principles from Section \ref{section max principles}.

The method of moving spheres is well-established on $\R^n$ and consists in comparing the values of a solution to some equation with its (suitably defined) inversion about a certain sphere $\partial B_\lambda(x_0) \subset \R^n$. Using stereographic projection, we lift this procedure to $\Sph$. Namely, for any solution to \eqref{equation 1} and $\lambda > 0$, $\xi_0 \in \Sph \setminus \{S\}$, we will compare $u$ on the set
$\Slxi = \Ste (B_\lambda(\Ste^{-1}(\xi_0)))$ with its reflected version $u_{\Phi_\lxi}$. Recall that the map $\Phi_\lxi$ has been introduced in \eqref{definition Phi lambda xi} and the definition of $u_\Phi$ has been given in \eqref{definition u Phi}.

At the same time we need to consider the reflection of $u$ about (stereographically projected) planes, i.e., $u_{\Psi_{\alpha,e}}$ for $e\in\mathbb S^{n-1}$, $\alpha\in\R$, with $\Psi_{\alpha,e}$ defined in \eqref{definition Psi a e}.

The following is the main result of this section.

\begin{theorem}
\label{theorem symmetry}
Let $u \geq 0$ be a weak solution to \eqref{equation 1}. Then the following holds.
\begin{enumerate}
\item[(i)]
For every $\xi_0 \in \Sph\setminus\{S\}$, there is a $\lambda_0 = \lambda_0(\xi_0) > 0$ such that $u_{\Phi_{\lambda_0,\xi_0}} \equiv u$.
\item[(ii)]
For every $e \in \mathbb S^{n-1}$, there is $a = a(e) \in \R$ such that $u_{\Psi_{\alpha,e}} \equiv u$.
\end{enumerate}
\end{theorem}

As in Section \ref{section max principles}, since the arguments to prove parts (i) and (ii) are very similar and of comparable difficulty, for clarity of exposition we focus in the following on proving part (i) of Theorem \ref{theorem symmetry}. The reader is invited to check that all arguments given in the rest of the present section remain valid when $\Phi_\lxi$ is replaced by $\Psi_{\alpha,e}$ and therefore yield a proof of part (ii) as well.

\subsection{The moving spheres argument}

In this subsection, we fix $\xi_0 \in \Sph \setminus \{S\}$ and let $\lambda > 0$ vary. We abbreviate
\[ u_\lxi := u_{\Phi_\lxi}. \]
We will prove Theorem \ref{theorem symmetry} by analyzing the positivity of the difference
\[ w_\lxi := u_\lxi - u  \]
on $\Slxi$. Since $\Phi_\lxi^2 = \text{id}_{\Sph \setminus \{\xi_0, S\}}$, the function $w_\lxi$ is antisymmetric with respect to $\Phi_\lxi$. By the conformal invariance proved in Lemma \ref{lemma conf inv}, both $u$ and $u_\lxi$ are weak solutions of \eqref{equation 1} and therefore the function $w_\lxi$ satisfies
\begin{equation}
\label{equality for w}
\mathcal E[\varphi,w_\lxi] = \int_{\Sph} \varphi(\xi) h(\xi) w_\lxi (\xi) \diff \xi
\qquad\text{for all}\ \varphi \in \mathcal D
\end{equation}
with
$$
h(\xi) := \begin{cases} 
\frac{g(u_\lxi(\xi))- g(u(\xi))}{u_\lxi(\xi)-u(\xi)} & \text{if}\ u_\lxi(\xi)\neq u(\xi) \,,\\
g'(u(\xi)) & \text{if}\ u_\lxi(\xi)= u(\xi) \,,
\end{cases}
\qquad\text{and}\qquad
g(u) := C_n u \ln u \,.
$$

Convexity of $g$ implies that
$$
h(\xi)w_\lxi (\xi) \geq g'(u(\xi)) w_\lxi (\xi)
\qquad\text{if}\ u_\lxi(\xi)\leq u(\xi)
$$
and a simple computation shows that
$$
h(\xi)w_\lxi (\xi) \geq - e^{-1}
\qquad\text{if}\ u_\lxi(\xi)\geq u(\xi) \,.
$$
Thus, setting
\[ \Sigma_\lxi^- := \{ \eta \in \Slxi \, : \, w_\lxi(\eta) <  0 \} ,\]
and
$$
V(\xi) := -g'(u(\xi))\ \1_{\Sigma_\lxi^-}(\xi) + (e w_\lxi(\xi))^{-1} \1_{\Sigma_\lxi\setminus\Sigma_\lxi^-}(\xi) \,,
$$
we have $h w_\lxi \geq - V w_\lxi$ on $\Sigma_\lxi$ and, consequently,
\begin{equation}
\label{inequality for w}
\mathcal E[\varphi,w_\lxi] + \langle \varphi, V w_\lxi \rangle \geq 0 \qquad \text{ for all } 0\leq\varphi\in\mathcal D \text{ with } \varphi = 0 \text{ on } (\Sigma_\lxi)^c \,.
\end{equation}

The first step in the method of moving spheres is the following application of the small volume maximum principle from Lemma \ref{lemma max principle narrow region}.

\begin{lemma}
[Starting the sphere]
\label{lemma starting the sphere}
Let $\xi_0 \in \Sph \setminus \{S \}$ be fixed. Then for every $\lambda > 0$ small enough, we have $w_\lxi \geq 0$ a.e.~on $\Slxi$.
\end{lemma}

\begin{proof}
	We will apply Lemma \ref{lemma max principle narrow region} with $\Omega = \Sigma_\lxi^-$. As remarked before, $w_\lxi$ is antisymmetric. Assumption \eqref{assumptions on w narrow region mp 1} follows from \eqref{inequality for w} and Assumption \eqref{assumptions on w narrow region mp 2} follows by definition of $\Omega = \Sigma_\lxi^-$. Finally,
	$$
	\int_\Omega e^{2V_-/C_n} = e^2\int_{\{u > e^{-1}\}\cap \Sigma_\lxi^-} u^2 \leq e^2 \int_{\Sigma_\lxi^-} u^2 \,.
	$$
	Since $\1_{\Sigma_\lxi^-} \to 0$ a.e.~as $\lambda\to 0$ and $u\in L^2(\Sph)$, we deduce from dominated convergence that
	$$
	\int_\Omega e^{2V_-/C_n} < |\Sph|
	\qquad\text{for all sufficiently small}\ \lambda>0 \,.
	$$
	Thus, Lemma \ref{lemma max principle narrow region} implies that $w_\lxi\geq 0$ a.e.~on $\Sigma^-_{\lambda, \xi_0}$, so $|\Sigma^-_{\lambda, \xi_0}|=0$, which is the assertion of the lemma.
\end{proof}

Due to Lemma \ref{lemma starting the sphere}, the  `critical scale' associated to $\xi_0$,
\begin{equation}
\label{definition critical scale}
\lambda_0(\xi_0) := \sup \left \{ \lambda >0 \,:\, w_{\mu, \xi_0}(\eta) \geq 0 \text{ for all } 0 < \mu < \lambda \text{ and almost every }  \eta \in{\Sigma_{\mu,\xi_{0}}} \right\},
\end{equation}
is well-defined with $\lambda_0(\xi_0) \in (0, \infty]$.

\begin{proof}[Proof of Theorem \ref{theorem symmetry}]
We recall that $\xi_0\in\Sph\setminus\{S\}$ is fixed.

First, let us prove $\lambda_{0}(\xi_{0})<\infty$ by contradiction. Assuming that $\lambda_{0}(\xi_{0})=+\infty$, we can choose $\lambda_{i}>0$ with $\lambda_{i}\rightarrow{+\infty}$
and $u_{\lambda_{i},\xi_{0}}-u=w_{\lambda_{i},\xi_{0}}\geq0$ a.e.~on $\Sigma_{\lambda_{i},\xi_{0}}$. Integrating over $\Sigma_{\lambda_i, \xi_0}$ and changing variables, we obtain
\[ \int_{\Sigma_{\lambda_i, \xi_0}} u(\eta)^2 \diff \eta \leq \int_{\Sigma_{\lambda_i, \xi_0}} J_{\lambda_i, \xi_0}(\eta) u(\Phi_{\lambda_i, \xi_0} \eta)^2 \diff \eta = \int_{\Sph \setminus \Sigma_{\lambda_i, \xi_0}} u(\eta)^2 \diff \eta \,, \]
that is,
$$
\int_{\Sph \setminus \Sigma_{\lambda_i, \xi_0}} u(\eta)^2 \diff \eta
\geq \frac 12 \int_\Sph u(\eta)^2\diff\eta \,.
$$
Since $\1_{\Sph \setminus \Sigma_{\lambda, \xi_0}}\to 0$ a.e.~as $\lambda\to 0$ and $u\in L^2(\Sph)$, dominated convergence implies that the left side tends to zero as $i\to\infty$. This contradicts the assumption $u\not\equiv 0$. Thus, we have shown that $\lambda_0 := \lambda_{0}(\xi_{0})<\infty$.

Next, we prove that $w_{\lambda_0, \xi_0} \geq 0$ a.e.~on $\Sigma_{\lambda_0, \xi_0}$. By continuity of the map $\lambda \mapsto w_\lxi$ into $L^2(\Sph)$, we have, up to a subsequence, that $w_{\lambda,\xi_{0}}\rightarrow w_{\lambda_{0},\xi_{0}}$ a.e.~on $\Sigma_{\lambda_0, \xi_0}$ as $\lambda \nearrow \lambda_0$ from below. Consequently, by the definition of $\lambda_{0}$ we have $w_{\lambda_0, \xi_0} \geq 0$ a.e.~on $\Sigma_{\lambda_0, \xi_0}$.

Next, we claim that either $w_{\lambda_0, \xi_0}  \equiv 0$ or $w_{\lambda_0, \xi_0}  > 0$ a.e.~on $\Sigma_{\lambda_{0},\xi_{0}}$. We will deduce this from Lemma \ref{lemma strong max principle}. Assumption \eqref{assumptions on w strong mp 1} follows from \eqref{inequality for w} and we have already verified assumption \eqref{assumptions on w strong mp 2}. Finally, $V_+ w_{\lambda_0,\xi_0} \leq e^{-1}$ is bounded. Therefore Lemma \ref{lemma strong max principle} is applicable and yields the claimed dichotomy.

Finally, in order to show that $w_{\lambda_0, \xi_0}  \equiv 0$, we argue by contradiction and assume that $w_{\lambda_{0}, \xi_0} > 0$ a.e.~on $\Sigma_{\lambda_{0},\xi_{0}}$. Similarly as in the proof of Lemma \ref{lemma starting the sphere} we choose $\Omega = \Sigma_\lxi^-$ and bound, for $\lambda>\lambda_0$,
$$
\int_\Omega e^{2V_-/C_n} \leq e^2 \int_{\Sigma_{\lambda_0,\xi_0}\cap\{ w_{\lambda,\xi_0}<0\}} u^2 + e^2 \int_{\Sigma_{\lambda,\xi_0}\setminus \Sigma_{\lambda_0,\xi_0}} u^2 \,.
$$
Since $w_{\lambda,\xi_{0}}\rightarrow w_{\lambda_0,\xi_{0}}$ a.e.~on $\Sigma_{\lambda_{0},\xi_{0}}$ as $\lambda\to \lambda_0$ and $w_{\lambda_0,\xi_{0}}>0$ a.e.~on $\Sigma_{\lambda_{0},\xi_{0}}$, we have $\1_{\{w_{\lambda,\xi_{0}}<0\}} \rightarrow 0$ a.e.~on $\Sigma_{\lambda_{0},\xi_{0}}$ as $\lambda\to \lambda_0$. Moreover, clearly, $\1_{\Sigma_{\lambda,\xi_0}\setminus \Sigma_{\lambda_0,\xi_0}}\to 0$ a.e.~as $\lambda\searrow\lambda_0$. By dominated convergence these facts, together with $u\in L^2(\Sph)$, imply that
$$
\int_\Omega e^{2V_-/C_n} <|\Sph|
\qquad\text{for all sufficiently small}\ \lambda-\lambda_{0}>0 \,.
$$
The small volume maximum principle from Lemma \ref{lemma max principle narrow region} therefore implies that $w_\lxi \geq 0$
a.e.~on $\Sigma_{\lambda,\xi_{0}}$ for all sufficiently small $\lambda-\lambda_0>0$. This contradicts the definition of $\lambda_0(\xi_0)$ from \eqref{definition critical scale} and therefore proves that $w_{\lambda_0, \xi_0} \equiv 0$, as claimed.
\end{proof}


\section{Proof of the main result}
\label{section proof main}

In this section we use the symmetry of $u$ derived in Theorem \ref{theorem symmetry} via the method of moving spheres in order to deduce that $u$ must be of the form claimed in Theorem~\ref{theorem classification}. This will be a consequence of the symmetry result of Li and Zhu \cite{LiZh} in the generalized form stated in \cite{FrLi}. Actually, the theorem in \cite{FrLi} is for arbitrary finite measures, but we shall only quote a version for the case of measures which are absolutely continuous with respect to Lebesgue measure; see the remark after \cite[Theorem 1.4]{FrLi}.

\begin{theorem}
[{\cite[Theorem 1.4]{FrLi}}]
\label{theorem li zhu}
Let $v \in L^2(\R^n)$ be nonnegative. Assume that for any $x_0 \in \R^n$ there is a $\lambda > 0$ such that
\begin{equation}
\label{li zhu inversion}
v(x) = v_{I_{\lambda, x_0}}(x) \qquad \text{ for almost every } \quad x \in \R^n
\end{equation}
and for any $e \in \mathbb S^{n-1}$ there is an $\alpha \in \R$ such that
\begin{equation}
\label{li zhu reflection}
v(x) = v_{R_{\alpha,e}}(x) \qquad \text{ for almost every } \quad x \in \R^n.
\end{equation}
Then there are $a \in \R^n$, $b > 0$ and $c \geq 0$ such that
\begin{equation}
\label{li zhu classification}
v(x) = c \left( \frac{2b}{b^2 + |x-a|^2} \right)^{n/2} \,.
\end{equation}
\end{theorem}

We can now give the proof of our main result.

\begin{proof}
[Proof of Theorem \ref{theorem classification}]
From Theorem \ref{theorem symmetry} we deduce immediately that the function $v = u_\Ste$ (in the notation of \eqref{definition u Phi}) satisfies the assumptions of Theorem \ref{theorem li zhu}. Therefore, $v$ is of the form \eqref{li zhu classification} for some $a \in \R^n$, $b > 0$ and $c \geq 0$. A computation shows that
$$
u(\omega) = c \left( \frac{\sqrt{1-|\zeta|^2}}{1-\zeta\cdot\omega} \right)^{n/2}
$$
with a certain $\zeta\in\R^{n+1}$ with $|\zeta|<1$ which is given explicitly in terms of $a$ and $b$; see the discussion after \eqref{eq:optsphere}. Thus, there is a conformal mapping $\Phi$ on $\Sph$ (corresponding via stereographic projection to translation by $a$ and dilation by $b$ on $\R^n$) such that
$$
u = c J_\Phi^{1/2} = c \1_\Phi \,.
$$
Here $\1$ is the function on $\Sph$ which is constant one and $\1_\Phi$ refers to notation \eqref{definition u Phi}. By equation \eqref{equation 1} and its conformal invariance given in \eqref{conf transf H}, we have
$$
C_n u \ln u = H u = c H \1_\Phi = c \left( (H\1)_\Phi + C_n \1_\Phi \ln J_\Phi^{1/2} \right) = C_n u \ln \frac{u}{c} \,.
$$
This implies $c=1$ and concludes the proof of the theorem.
\end{proof}



\begin{thebibliography}{21}

\bibitem{Be1992} W.~Beckner, \textit{Sobolev inequalities, the Poisson semigroup, and analysis on the sphere $\Sph$}. Proc. Nat. Acad. Sci. U.S.A. \textbf{89} (1992), no. 11, 4816--4819.

\bibitem{Be1993} W. Beckner, \textit{Sharp Sobolev inequalities on the sphere and the Moser-Trudinger inequality}. Ann. of Math. (2) \textbf{138} (1993), no. 1, 213–242

\bibitem{Be1997} W. Beckner, \textit{Logarithmic Sobolev inequalities and the existence of singular integrals}.
Forum Math. \textbf{9} (1997), no. 3, 303--323.

\bibitem{CaGiSp} L.~Caffarelli, B.~Gidas, J.~Spruck, \textit{Asymptotic symmetry and local behavior of semilin- ear equations with critical Sobolev growth}. Comm. Pure Appl. Math. \textbf{42} (1989), 271--297.

\bibitem{ChWe} H. Chen, T. Weth, \textit{The Dirichlet problem for the logarithmic Laplacian}. Comm. Partial Differential Equations \textbf{44} (2019), no. 11, 1100--1139.

\bibitem{ChLiLi} W. Chen, C. Li, Y. Li, \textit{A direct method of moving planes for the fractional Laplacian}. Adv. Math. \textbf{308} (2017), 404–437.

\bibitem{ChLiOu} W. Chen, C. Li, B. Ou, \textit{Classification of solutions for an integral equation}. Comm. Pure Appl. Math. \textbf{59} (2006), no. 3, 330--343.

\bibitem{ChLiZh} W. Chen, Y. Li, R. Zhang, \textit{A direct method of moving spheres on fractional order equations}. J. Funct. Anal. \textbf{272} (2017), no. 10, 4131–4157

\bibitem{DLNoPo} L. De Luca, M. Novaga, M. Ponsiglione, \textit{The 0-fractional perimeter between fractional perimeters and Riesz potentials}. Preprint (2019), arXiv:1906.06303.

\bibitem{FrLe} R. L. Frank, E. Lenzmann, \textit{On ground states for the $L^2$-critical boson star equation}. Preprint (2010), arXiv:0910.2721.

\bibitem{FrLi} R. L. Frank, E. H. Lieb, \textit{Inversion positivity and the sharp Hardy-Littlewood-Sobolev inequality}. Calc. Var. Partial Differential Equations \textbf{39} (2010), no. 1-2, 85–99. 

\bibitem{FrLi2012} R. L. Frank, E. H. Lieb, \textit{A new, rearrangement-free proof of the sharp Hardy-Littlewood-Sobolev inequality}. Spectral theory, function spaces and inequalities, 55–67, Oper. Theory Adv. Appl. \textbf{219}, Birkhäuser/Springer Basel AG, Basel, 2012. 

\bibitem{GiNiNi} B.~Gidas, W.~M.~Ni, L.~Nirenberg, \textit{Symmetry of positive solutions of nonlinear elliptic equations in $\R^n$}. In: Mathematical Analysis and Applications, Adv. in Math. Suppl. Stud. \textbf{7A} (1981), 369--402.

\bibitem{GrJeMaSp} C.~R.~Graham, R.~Jenne, L.~J.~Mason, G.~A.~Sparling, \textit{Conformally invariant powers of the Laplacian. I. Existence}. J. London Math. Soc. (2) \textbf{46} (1992), no. 3, 557--565. 

\bibitem{JaWe} S.~Jarohs, T.~Weth, \textit{Symmetry via antisymmetric maximum principles in nonlocal problems of variable order}. Ann. Mat. Pura Appl. (4) \textbf{195} (2016), no. 1, 273–291.

\bibitem{JiXi} T.~Jin, J.~Xiong, \textit{A fractional Yamabe flow and some applications}. J.~reine angew.~Math. \textbf{696} (2014), 187--223.

\bibitem{Li04} Y.~Y.~Li, \textit{Remark on some conformally invariant integral equations: the method of moving spheres}. J. Eur. Math. Soc. (JEMS) \textbf{6} (2004), no. 2, 153--180.

\bibitem{LiZha} Y.~Y.~Li, L.~Zhang, \textit{Liouville type theorems and Harnack type inequalities for semilinear elliptic equations}. J. Anal. Math. \textbf{90} (2003), 27--87.

\bibitem{LiZh} Y.~Y.~Li, M.~Zhu, \textit{Uniqueness theorems through the method of moving spheres}. Duke Math. J. \textbf{80} (1995), no. 2, 383--417.

\bibitem{Li} E.~H.~Lieb, \textit{Sharp constants in the Hardy--Littlewood--Sobolev and related inequalities}. Ann. of Math. (2) \textbf{118} (1983), no. 2, 349--374.

\bibitem{MaZh} L.~Ma, L.~Zhao, \textit{Classification of positive solitary solutions of the nonlinear Choquard equation}. Arch. Ration. Mech. Anal. \textbf{195} (2010), 455--467.

\end{thebibliography}
\end{document}